\documentclass{amsart}

\begin{document}
\newtheorem{definition}{Definition}[section]
\newtheorem{example}[definition]{Example}
\newtheorem{remark}[definition]{Remark}
\newtheorem{observation}[definition]{Observation}
\newtheorem{theorem}[definition]{Theorem}
\newtheorem{lemma}[definition]{Lemma}
\newtheorem{proposition}[definition]{Proposition}
\newtheorem{corollary}[definition]{Corollary}

%
%
%
%
\title{Compact fuzzy soft spaces}{}
\author{\.Ismail Osmano\u{g}lu, Deniz Tokat}{}

\maketitle

\vspace{-3mm}
\noindent\hrulefill \\
\vspace{-3mm}

\begin{abstract}
In this article, by using basic properties of fuzzy soft topology we defined fuzzy soft compactness. We also
introduced some basic definitions and theorems of the concept.

2010 AMS Classification: 54A40, 03E72, 54D30.

Keywords: Fuzzy soft topology, Fuzzy soft cover, Fuzzy soft compactness.
\\
\end{abstract}
\vspace{-3mm}

\noindent\hrulefill \\

\section{Introduction}

 
 There are many problems encountered in real life. Various mathematical set theories such as soft set which was introduced by Molodtsov \cite{7} and fuzzy set which developed by Zadeh \cite{10} have been developed to solve these problems. P. K. Maji, R. Biswas, A. R. Roy \cite{11} also initiated the more generalized concept of fuzzy soft sets which is a combination of fuzzy set and soft set. Then many researchers have applied this concept. B. Tanay et. al. introduced topological structure of fuzzy soft set in \cite{5} and gave a introductory theoretical base to carry further study on this concept. Following this study, some others (\cite{6},\cite{9},\cite{12},\cite{13}) studied on the concept of fuzzy soft topological spaces. We will introduce compactness on fuzzy soft topological spaces and give some important definitions and theorems.

\section{Preliminaries}

\begin{definition} \cite{10}
A fuzzy set $A$ of a non-empty set $X$ is characterized by a membership
function $\mu _{A}:X\rightarrow \left[ 0,1\right] $ whose value $\mu _{A}\left( x\right) $ represents the "grade of membership" of $x$ in $%
A $ for $x\in X$.

Let $I^{X}$ denotes the family of all fuzzy sets on $X$. If $A,B\in I^{X}$,
then some basic set operations for fuzzy sets are given by Zadeh as follows:

\begin{description}
\item[$\left( 1\right) $] $A\leq B\Leftrightarrow \mu _{A}\left( x
\right) \leq \mu _{B}\left( x \right) $, for all $x\in X$.

\item[$\left( 2\right) $] $A=B\Leftrightarrow \mu _{A}\left( x \right)
=\mu _{B}\left( x \right) $, for all $x\in X$.

\item[$\left( 3\right) $] $C=A\vee B\Leftrightarrow \mu _{C}\left( x
\right) =\mu _{A}\left( x \right) \vee \mu _{B}\left( x \right) $,
for all $x\in X$.

\item[$\left( 4\right) $] $D=A\wedge B\Leftrightarrow \mu _{D}\left( x
\right) =\mu _{A}\left( x \right) \wedge \mu _{B}\left( x \right) $%
, for all $x\in X$.

\item[$\left( 5\right) $] $E=A^{c}\Leftrightarrow \mu _{E}\left( x
\right) =1-\mu _{A}\left( x \right) $, for all $x\in X$.
\end{description}
\end{definition}

\begin{definition}\cite{7}
Let $X$ be the initial universe set and $E$ be the set of
parameters. A pair $(F,A)$ is called a soft set over $X$ where $F$ is a
mapping given by $F:A\rightarrow P\left( X\right) $ and $A\subseteq E$.

In the other words, the soft set is a parametrized family of subsets of the
set $X$. Every set $F(e)$, for every $e\in A$, from this family may be
considered as the set of $e$-elements of the soft set $(F,A)$.
\end{definition}

\begin{definition}\cite{11}
Let $A\subseteq E$. A pair $(f,A)$ is called a fuzzy soft set
over $X$ , where $f:A\rightarrow I^{X}$ is a function.

That is, for each $a\in A$, $f\left( a\right) =f_{a}:X\rightarrow I$ is a
fuzzy set on $X$.
\end{definition}

\begin{definition}\cite{6}
Fuzzy soft set $(f,A)$ on the universe $X$ is a mapping from the parameter
set $E$ to $I^{X}$, i.e., $(f,A):E$ $\rightarrow $ $I^{X}$, where $%
(f,A)\left( e\right) \neq 0_{X}$ if $e\in A\subseteq E$ and $(f,A)\left(
e\right) =0_{X}$ if $e\notin A$, where $0_{X}$ is empty fuzzy set on $X$.

From now on, we will use $FS(X,E)$ instead of the family of all fuzzy soft
sets over $X$.
\end{definition}

\begin{definition}\cite{6}
Let $\left( f,A\right) ,\left( g,B\right) \in FS(X,E)$. The following
operations are defined as follows:

\begin{description}
\item[Subset] $\left( f,A\right) \tilde{\subseteq}\left( g,B\right) $ if $%
\left( f,A\right) \left( e\right) \leq \left( g,B\right) \left( e\right) ,$
for each $e\in E.$

\item[Equal] $\left( f,A\right) =\left( g,B\right) $ if $\left(
f,A\right) \tilde{\subseteq}\left( g,B\right) $ and $\left( g,B\right) 
\tilde{\subseteq}\left( f,A\right) .$

\item[Union] $\left( h,A\cup B\right) =\left( f,A\right) \tilde{\cup}
\left( g,B\right) $ where $\left( h,A\cup B\right) \left( e\right) =\left(
f,A\right) \left( e\right) \vee \left( g,B\right) \left( e\right) ,$ for all 
$e\in E.$

\item[Intersection] $\left( h,A\cap B\right) =\left( f,A\right) 
\tilde{\cap}\left( g,B\right) $where $\left( h,A\cap B\right) \left(
e\right) =\left( f,A\right) \left( e\right) \wedge \left( g,B\right) \left(
e\right) ,$ for all $e\in E.$
\end{description}
\end{definition}

\begin{definition}\cite{6}
Let $\left( f,A\right) \in FS(X,E)$. Then complement of $\left(
f,A\right) $, denoted by $\left( f,A\right) ^{c}$, is the fuzzy soft set
defined by $\left( f,A\right) ^{c}\left( e\right) =1_{X}-\left( f,A\right)
\left( e\right) $, for all $e\in E.$

Clearly $\left( \left( f,A\right) ^{c}\right) ^{c}=\left( f,A\right) .$
\end{definition}

\begin{definition}\cite{6}
Let $\left( f,E\right) \in FS(X,E)$.The fuzzy soft set $\left( f,E\right) $
is called the null fuzzy soft set, denoted by $\tilde{0}_{E}$, if $\left(
f,E\right) \left( e\right) =0_{X},$for all $e\in E.$
\end{definition}

\begin{definition}\cite{6}
Let $\left( f,E\right) \in FS(X,E)$.The fuzzy soft set $\left( f,E\right) $
is called theuniversal fuzzy soft set, denoted by $\tilde{1}_{E}$, if $%
\left( f,E\right) \left( e\right) =1_{X},$for all $e\in E.$

Clearly $\left( \tilde{1}_{E}\right) ^{c}=\tilde{0}_{E}$ and $\left( \tilde{0%
}_{E}\right) ^{c}=\tilde{1}_{E}.$
\end{definition}

\begin{definition}\cite{2}
Let $FS(X,E)$ and $FS(Y,K)$ be the families of all fuzzy soft
sets over $X$ and $Y$ , respectively. Let $\varphi :X\rightarrow Y$ and $%
\psi :E\rightarrow K$ be two functions. Then the pair $\left( \varphi ,\psi
\right) $ is called a fuzzy soft mapping from $X$ to $Y$, and denoted by $%
\left( \varphi ,\psi \right) :FS(X,E)\rightarrow FS(Y,K)$.

If $\varphi $ and $\psi $ is injective then the fuzzy soft mapping $\left(
\varphi ,\psi \right) $ is said to be injective. If $\varphi $ and $\psi $\
is surjective then the fuzzy soft mapping $\left( \varphi ,\psi \right) $ is
said to be surjective.

The fuzzy soft mapping $\left( \varphi ,\psi \right) $ is called constant,
if $\varphi $ and $\psi $\ are constant.
\end{definition}

\begin{definition}\cite{5}
A fuzzy soft topological space is a pair $\left( X,\tau \right) $
where $X$ is a nonempty set and $\tau $ a family of fuzzy soft sets over $X$
satisfying the following properties:

$(1)$ $\tilde{0}_{E},\tilde{1}_{E}\in \tau ,$

$(2)$ If $\left( f,A\right) ,\left( g,B\right) \in \tau ,then$ $\left(
f,A\right) \tilde{\cap}\left( g,B\right) \in \tau $

$(3)$ If $\left( f_{i},A\right) \in \tau ,i\in J,$ then $\cup _{i\in
I}\left( f_{i},A\right) \in \tau $

$\tau $ is called a topology of fuzzy soft sets on $X$. Every member of $%
\tau $ is called fuzzy soft open.

$\left( g,B\right) $ is called fuzzy soft closed in $\left( X,\tau \right) $
if $\left( g,B\right) ^{c}\in $ $\tau $.
\end{definition}

\begin{definition}\cite{5}
Let $(X,\tau _{1})$ and $(Y,\tau _{2})$ be two fuzzy soft topological
spaces. If each $\left( f,A\right) \in \tau _{1}$ is in $\tau _{2}$, then $%
\tau _{2}$ is called fuzzy soft finer than $\tau _{1}$, or (equivalently) $%
\tau _{1}$ is fuzzy soft coarser than $\tau _{2}$.
\end{definition}

\begin{definition}\cite{6}
Let $(X,\tau _{1})$ and $(Y,\tau _{2})$ be two fuzzy soft topological spaces.

$(1)$ A fuzzy soft mapping $\left( \varphi ,\psi \right) :(X,\tau
_{1})\rightarrow (Y,\tau _{2})$ is called fuzzy soft continuous if $\left(
\varphi ,\psi \right) ^{-1}\left( \left( g,B\right) \right) \in \tau _{1}$, $%
\forall \left( g,B\right) \in \tau _{2}.$

$(2)$ A fuzzy soft mapping $\left( \varphi ,\psi \right) :(X,\tau
_{1})\rightarrow (Y,\tau _{2})$ is called fuzzy soft open if $\left( \varphi
,\psi \right) \left( \left( f,A\right) \right) \in \tau _{2}$, $\forall
\left( f,A\right) \in \tau _{1}.$
\end{definition}

\section{Compact fuzzy soft spaces}

\begin{definition}
A family $\Psi $\ of fuzzy soft sets is a cover of a fuzzy soft set $\left(
f,A\right) $ if

$\qquad \left( f,A\right) \subseteq \cup \left\{ \left( f_{i},A\right)
:\left( f_{i},A\right) \in \Psi ,i\in I\right\} $.

It is a fuzzy soft open cover if each member of $\Psi $\ is a fuzzy soft
open set. A subcover of $\Psi $\ is a subfamily of $\Psi $\ which is also a
cover.
\end{definition}

\begin{definition}
Let $(X,\tau )$ be fuzzy soft topological space and $\left( f,A\right) \in
FS(X,E)$. Fuzzy soft set $\left( f,A\right) $ is called compact\ if\ each
fuzzy soft open cover of $\left( f,A\right) $ has a finite subcover. Also
fuzzy soft topological space $(X,\tau )$ is called compact if each fuzzy
soft open cover of $\tilde{1}_{E}$ has a finite subcover.
\end{definition}

\begin{example}
A fuzzy soft topological space $(X,\tau )$ is compact if $X$ is finite.
\end{example}

\begin{example}
Let $(X,\tau )$ and $(Y,\sigma )$ be two fuzzy soft topological spaces and $\tau
\subset \sigma $. Then, fuzzy soft topological space $(X,\tau )$ is compact
if $(Y,\sigma )$ is compact.
\end{example}

\begin{proposition}\label{p1}
Let $\left( g,B\right) $ be a fuzzy soft closed set in fuzzy\ soft\ compact
space $(X,\tau )$. Then $\left( g,B\right) $ is also compact.
\end{proposition}

\begin{proof}
Let $\left( f_{i},A\right) $ be any open covering of $\left( g,B\right) $.
Then $\tilde{1}_{X}\subseteq \left( \cup _{i\in I}\left( f_{i},A\right) \right)
\cup \left( g,B\right) ^{c}$, that is, $\left( f_{i},A\right) $ together with fuzzy soft open set $\left(
g_{B}\right) ^{c}$ is a open covering of $\tilde{1}_{X}$. Therefore there
exists a finite subcovering $\left( f_{1},A\right) ,\left( f_{2},A\right)
,...,\left( f_{n},A\right) ,\left( g,B\right) ^{c}$. Hence we obtain $\tilde{1}_{X}\subseteq \left( f_{1},A\right) \cup \left( f_{2},A\right)
\cup ...\cup \left( f_{n},A\right) \cup \left( g,B\right) ^{c}$.
Therefore, we get  $\left( g,B\right) \subseteq \left( f_{1},A\right) \cup \left(
f_{2},A\right) \cup ...\cup \left( f_{n},A\right) \cup \left( g,B\right)
^{c} $ which clearly implies  $\left( g,B\right) \subseteq \left( f_{1},A\right) \cup \left(
f_{2},A\right) \cup ...\cup \left( f_{n},A\right) $
since $\left( g,B\right) \cap \left( g,B\right) ^{c}=\Phi $. Hence $\left(
g,B\right) $ has a finite subcovering and so is compact.
\end{proof}

\begin{definition}\cite{9}
Let $(X,\tau )$ be a fuzzy soft topological space over $X$ and $x,y\in X$
such that $x\neq y$. If there exist fuzzy soft open sets $(f,A)$ and $(g,A)$
such that $x\in (f,A),y\in (g,A)$ and $(f,A)\tilde{\cap}(g,A)=\Phi $, then $%
(X,\tau )$ is called a fuzzy soft Hausdorff space.
\end{definition}

\begin{proposition}
Let $\left( g,B\right) $ be a fuzzy soft compact set in fuzzy\ soft\
Hausdorff space $(X,\tau )$. Then $\left( g,B\right) $ is closed.
\end{proposition}

\begin{proof}
Let $x\in \left( g,B\right) ^{c}$. For each $y\in \left( g,B\right) $, we
have $x\neq y$, so there are disjoint fuzzy soft open sets $\left( f_{y},A\right) $ and 
$\left( h_{y},A\right) $ so that $x\in \left( f_{y},A\right) $ and $y\in
\left( h_{y},A\right) $. Then $\{\left( h_{y},A\right) :y\in \left(
g,B\right) \}$ is an fuzzy soft open cover of $\left( g,B\right) $. Let $\{\left(
h_{y_{1}},A\right) ,\left( h_{y_{2}},A\right) ,...,\left( h_{y_{n}},A\right)
\}$ be a finite subcover. Then $\cap _{i=1}^{n}\left( f_{y_{i}},A\right) $
is an open set containing $x$ and contained in $\left( g,B\right)^{c}$.
Thus $\left( g,B\right) ^{c}$ is fuzzy soft open and $\left( g,B\right) $ is closed.
\end{proof}

\begin{theorem}
Let $(X,\tau )$ and $(Y,\sigma )$ be fuzzy soft topological spaces and $%
\left( \varphi ,\psi \right) :(X,\tau )\rightarrow (Y,\sigma )$ continuous
and onto fuzzy soft function. If $ (X,\tau )$ is fuzzy soft compact, then $(Y,\sigma )$ is fuzzy soft compact,
\end{theorem}

\begin{proof}
We will use Theorem 3.8. and Theorem 3.10. of \cite{3}. Let $\left( f_{i},A\right) $
be any open covering of $\tilde{1}_{Y}$, i.e., $\tilde{1}_{Y}\subseteq
\cup _{i\in I}\left( f_{i},A\right) $. Then $\left( \varphi ,\psi \right)
^{-1}\left( \tilde{1}_{Y}\right) \subseteq \left( \varphi ,\psi \right)
^{-1}\left( \cup _{i\in I}\left( f_{i},A\right) \right) $ and $\tilde{1}%
_{X}\subseteq \cup _{i\in I}\left( \varphi ,\psi \right) ^{-1}\left( \left(
f_{i},A\right) \right) $. So $\left( \varphi ,\psi \right) ^{-1}\left(
\left( f_{i},A\right) \right) $ is an open covering of $\tilde{1}_{X}$. As $%
(X,\tau )$ is compact, there are $1,2,...,n$ in $I$ such that

$\tilde{1}_{X}\subseteq \left( \varphi ,\psi \right) ^{-1}\left( \left(
f_{1},A\right) \right) \cup \left( \varphi ,\psi \right) ^{-1}\left( \left(
f_{2},A\right) \right) \cup ...\cup \left( \varphi ,\psi \right) ^{-1}\left(
\left( f_{n},A\right) \right) $.

Since $\left( \varphi ,\psi \right) $ is surjective, we have

\begin{tabular}{ll}
$\tilde{1}_{Y}$ & $=\left( \varphi ,\psi \right) \left( \tilde{1}_{X}\right) 
$ \\ 
& $\subseteq \left( \varphi ,\psi \right) \left( \left( \varphi ,\psi
\right) ^{-1}\left( \left( f_{1},A\right) \right) \cup ...\cup \left(
\varphi ,\psi \right) ^{-1}\left( \left( f_{n},A\right) \right) \right) $ \\ 
& $=\left( \varphi ,\psi \right) \left( \left( \varphi ,\psi \right)
^{-1}\left( \left( f_{1},A\right) \right) \right) \cup ...\cup \left(
\varphi ,\psi \right) \left( \left( \varphi ,\psi \right) ^{-1}\left( \left(
f_{n},A\right) \right) \right) $ \\ 
& $=\left( f_{1},A\right) \cup \left( f_{2},A\right) \cup ...\cup \left(
f_{n},A\right) $.%
\end{tabular}

So we have $\tilde{1}_{Y}\subseteq \left( f_{1},A\right) \cup \left(
f_{2},A\right) \cup ...\cup \left( f_{n},A\right) $, i.e., $\tilde{1}_{Y}$
is covered by a finite number of $\left( f_{i},A\right) $.

Hence $(Y,\sigma )$ is compact.
\end{proof}

\begin{definition}
Let $(X,\tau )$ and $(Y,\sigma )$ be two fuzzy soft topological spaces. A
fuzzy soft mapping $\left( \varphi ,\psi \right) :(X,\tau )\rightarrow
(Y,\sigma )$ is called fuzzy soft closed if $\left( \varphi ,\psi \right)
\left( \left( f,A\right) \right) $ is fuzzy soft closed set in $(Y,\sigma )$,
 for all\ fuzzy soft closed set $\left( f,A\right) $ in $(X,\tau ).$
\end{definition}

\begin{theorem}
Let $(X,\tau )$ be a fuzzy soft topological space and $(Y,\sigma )$ be a
fuzzy soft Hausdorff space. Fuzzy soft mapping $\left( \varphi ,\psi \right) 
$ is closed if fuzzy soft mapping $\left( \varphi ,\psi \right) :(X,\tau
)\rightarrow (Y,\sigma )$ is continuous.
\end{theorem}

\begin{proof}
Let $\left( g,B\right) $ be any fuzzy soft closed set in $(X,\tau )$. By
theorem \ref{p1} we have $\left( g,B\right) $ is compact. Since fuzzy soft mapping $\left(
\varphi ,\psi \right) $ is continuous, fuzzy soft set $\left( \varphi ,\psi
\right) \left( \left( g,B\right) \right) $ is compact in $(Y,\sigma )$. As $%
(Y,\sigma )$ is fuzzy soft Hausdorff space, fuzzy soft set $\left( \varphi
,\psi \right) \left( \left( g,B\right) \right) $ is closed. Then Fuzzy soft
mapping $\left( \varphi ,\psi \right) $ is closed.
\end{proof}

\begin{definition}
A family $\Psi $\ of fuzzy soft sets has the finite intersection property if
the intersection of the members of each finite subfamily of $\Psi $\ is not the
null fuzzy soft set.
\end{definition}

\begin{theorem}
A fuzzy soft topological space is compact if and only if each family of
fuzzy soft closed sets with the finite intersection property has a nonnull
intersection.
\end{theorem}

\begin{proof}
$\Rightarrow :$ Let $\Psi $ be any family of fuzzy soft closed subset such that $%
\cap \{\left( f_{i},A\right) :\left( f_{i},A\right) \in \Psi ,i\in I\}=%
\tilde{0}_{E}$. Consider $\Omega =\{\left( f_{i},A\right) ^{c}:\left(
f_{i},A\right) \in \Psi ,i\in I\}$. So $\Omega $ is a fuzzy soft open cover
of $\tilde{1}_{E}$. As fuzzy soft topological space is compact, there exists
a finite subcovering $\left( f_{1},A\right) ^{c},\left( f_{2},A\right)
^{c},...,\left( f_{3},A\right) ^{c}$. Then $\cap _{i=1}^{n}\left(
f_{i},A\right) =\tilde{1}_{E}-\cup _{i=1}^{n}\left( f_{i},A\right) ^{c}=%
\tilde{1}_{E}-\tilde{1}_{E}=\tilde{0}_{E}$. Hence $\Psi $ can not have finite
intersection property.

$\Leftarrow :$ Assume that a fuzzy soft topological space is not compact. Then any fuzzy soft open cover of $\tilde{1}_{E}$ has not a finite subcover. Let $\{\left(f_{i},A\right) :i\in I\}$ be fuzzy soft open cover of $\tilde{1}_{E}$. So $ \cup _{i=1}^{n}\left( f_{i},A\right) \neq \tilde{1}_{E}$. Therefore $\cap_{i=1}^{n}\left( f_{i},A\right) ^{c}\neq \tilde{0}_{E}$. Thus, $\left\{ \left(f_{i},A\right) ^{c}:i=1,...,n\right\} $ have finite intersection property. By using hypothesis,  $\cap \left( f_{i},A\right) ^{c} \\ \neq \tilde{0}_{E}$ and we have  $\cup \left( f_{i},A\right) \neq \tilde{1}_{E}$. This is a contradiction. Thus the fuzzy soft topological space is compact.
\end{proof}

\section{Conclusion}
In this work, we introduced fuzzy soft compactness and gave basic definitions and theorems of this concept. Also we introduced fuzzy soft cover, fuzzy soft subcover, fuzzy soft open cover.

   {\sc\underline{\.Ismail Osmano\u{g}lu}} ({\tt ismailosmanoglu@yahoo.com}) --
      Department of Mathematics, Faculty of Arts and Sciences, Nev\c{s}ehir University, Nev\c{s}ehir, Turkey.\medskip

\noindent{\sc\underline{Deniz Tokat}} ({\tt dtokat@nevsehir.edu.tr}) --
       Department of Mathematics, Faculty of Arts and Sciences, Nev\c{s}ehir University, Nev\c{s}ehir, Turkey.

\end{document}